\documentclass[11pt]{amsart}
\usepackage{latexsym,amsfonts,amssymb}
\usepackage{hyperref}

\newtheorem{theorem}{Theorem}[section]
\newtheorem{lemma}[theorem]{Lemma}
\newtheorem{corollary}[theorem]{Corollary}
\newtheorem{question}[theorem]{Question}
\theoremstyle{definition}
\newtheorem{definition}[theorem]{Definition}

\theoremstyle{remark}

\newtheorem{example}[theorem]{Example}

\begin{document}

\title{Regular Bases At Non-isolated Points And Metrization Theorems}

\author{Fucai Lin}
\address{Fucai Lin: Department of Mathematics and Information Science,
Zhangzhou Normal University, Zhangzhou 363000, P. R. China}
\email{linfucai2008@yahoo.com.cn}
\author{Shou Lin}
\address{Shou Lin: Institute of Mathematics, Ningde Teachers' College, Ningde, Fujian
352100, P. R. China}
\email{linshou@public.ndptt.fj.cn}
\author{Heikki J. K. Junnila}
\address{Heikki Junnila: Department of Mathematics, University of Helsinki,
Yliopistonkatu 5, Helsinki 10, Finland}
\email{heikki.junnila@helsinki.fi}
\thanks{Supported by the NSFC(No. 10971185) and the Educational Department of Fujian Province(No. JA09166).}


\begin{abstract} In this paper, we define the spaces with a regular
base at non-isolated points and discuss some metrization theorems.
We firstly show that a space $X$ is a metrizable space, if
and only if $X$ is a regular space with a $\sigma$-locally finite
base at non-isolated points, if and only if $X$ is a perfect space
with a regular base at non-isolated points, if and only if $X$ is a
$\beta$-space with a regular base at non-isolated points. In
addition, we also discuss the relations between the spaces with a
regular base at non-isolated points and some generalized metrizable
spaces. Finally, we give an affirmative answer for a question posed by F. C. Lin and S. Lin in \cite{LL}, which also shows that a space with a regular base at non-isolated points has a point-countable base.
\end{abstract}

\maketitle

\parskip 0.15cm

\section{Introduction}
The bases of topological spaces occupy a core position in the study
of the topological theories and metrization problems, which has
produced many kinds of metrization theorems, and establishes a
foundation for the topological development \cite{Ho}. For example,
the following is a classic metrization theorem.

\begin{theorem}\label{t2}
The following are equivalent for a space $X$:
\begin{enumerate}
\item $X$ is metrizable;

\item $X$ is a $\mbox{T}_{1}$-space with a regular base;

\item $X$ is a regular space with a $\sigma$-locally finite base.
\end{enumerate}
\end{theorem}

In recent years, the theory of regular bases in topological spaces
played an important role in generalized metrizable spaces
\cite{AJRS, STY}. On the other hand, in the study of the theories of
topological spaces, we are mainly concerned with the properties of
neighborhoods on non-isolated points, and also discuss the relation
between their properties and global properties. For example, a study
of spaces with a sharp base, a weakly uniform base or an uniform
base at non-isolated points \cite{AJRS, Au, LL} shows that some
properties of a non-isolated point set of a topological space will
help us discuss the global construction of a space. Especially, a space
$X$ with a uniform base at non-isolated points if and only if $X$
is the open and boundary-compact image of a metric space \cite{LL}.
The most typical example is
the spaces obtained from a metrizable space by isolating the points of a subset.

Let $\mathcal{B}$ be a base for a space $X$. For any $x\in X$,
the base $\mathcal{B}$ of $X$ is called {\it regular} at a point $x$ if, for every
neighborhood $U$ of $x$, there exists an open subset $V$ such that
$x\in V\subset U$ and $\{B\in\mathcal{B}: B\cap
V\neq\emptyset\mbox{~and~}B\not\subset U\}$ is finite.

By Theorem~\ref{t2},  every metric space has a base which is regular at
non-isolated points. However, there exists a non-metrizable space
with a base which is regular at non-isolated points, see the
following Example~\ref{e0}.

\begin{example}\label{e0}
Let $X$ be the closed unit interval $\mathbb I=[0, 1]$ and $B$ a
Bernstein subset of $I$. In other words, $B$ is an uncountable set
which contains no uncountable closed subset of $I$. Endow $X$ with
the following topology, i.e., Michael line \cite{EA}: $G$ is an open
subset for $X$ if and only if $G=U\cup Z$, where $U$ is an open
subset of $\mathbb I$ with Euclidean topology and $Z\subset B$. Let
$\mathcal{B}$ be a base of $\mathbb I$ with the Euclidean topology,
where $\mathcal{B}$ is regular at every point of $\mathbb I$. Then
$\mathcal{P}=\mathcal{B}\cup\{\{x\}:x\in B\}$ is a base for $X$ and
also regular at non-isolated points.
\end{example}

Hence this causes our interests in a study of spaces with a base
which is regular at non-isolated points, and the related problems of
the metrizability. In this paper, we shall prove that spaces with a
regular base at non-isolated points are strictly between the
discretizations of
metrizable spaces and proto-metrizable
spaces, and we also obtain some metrization theorems which help us
to better understand the relation between the properties at
non-isolated points and global properties in the study the
generalized metrizable spaces.

In this paper all spaces are $\mbox{T}_{1}$ unless it is explicitly
stated which separation axiom is assumed, and all maps are
continuous and onto. By $\mathbb{R, N}$, denote the set of real
numbers and positive integers, respectively. For a space $X$, let
$I=I(X)=\{x:x \mbox{ is an isolated point of } X\}$ and $\mathcal
I(X)=\{\{x\}:x\in I(X)\}$. Let $\mathcal{P}$ be a family of subsets
for $X$, and we denote

\hspace{2cm}$\mbox{st}(x, \mathcal{P})=\cup\{P\in\mathcal{P}: x\in
P\}, x\in X$;

\hspace{2cm}$\mbox{st}(A, \mathcal{P})=\cup\{P\in\mathcal{P}: A\cap
P\not=\emptyset\}, A\subset X$;

\hspace{2cm}$\mathcal{P}^{m}=\{P\in\mathcal{P}:\mbox{~if~}P\subset
Q\in\mathcal{P}, \mbox{~then~}Q=P\}$.\

Readers may refer to \cite{ER, Ls} for unstated definitions and
terminology.

\section{Regular Bases at non-isolated points}
\begin{definition}
Let $\mathcal{B}$ be a base of a space $X$. $\mathcal{B}$ is a {\it
regular base}, see e.g.~\cite{ER}~({\it regular base at non-isolated points},
resp.) for $X$ if for each (non-isolated, resp.) point $x\in X$,
$\mathcal{B}$ is regular at $x$.
\end{definition}

It is obvious that regular bases
$\Rightarrow$ regular bases at non-isolated points, but regular
bases at non-isolated points$\nRightarrow$ regular bases by Example~\ref{e0}.

\begin{definition}\label{d0}
Let $\{\mathcal{W}_{i}\}_{i\in\mathbb{N}}$ be a sequence of open
covers of a space $X$ and $\mathcal I(X)\subset
\bigcup_{i\in\mathbb{N}}\mathcal{W}_{i}$.
$\{\mathcal{W}_{i}\}_{i\in\mathbb{N}}$ is called a {\it strong
development}, see e.g.~\cite{ER}({\it strong development at non-isolated
points}, resp.) for $X$ if for every $x\in X$~($x\in X-I$) and each
neighborhood $U$ of $x$ there exist a neighborhood $V$ of $x$ and an
$i\in \mathbb{N}$ such that $\mbox{st}(V,\mathcal{W}_{i})\subset
U$. If $\{\mathcal{W}_{i}\}_{i\in\mathbb{N}}$ is a strong development at
non-isolated points, then so is $\{\mathcal{W}_{i}\cup \mathcal
I(X)\}_{i\in\mathbb{N}}$.
\end{definition}

The following Lemma~\ref{l0} is proved similarly to Lemma
5.4.3 in \cite{ER}, and leave to the reader the easy proofs of Lemma~\ref{l1} and~\ref{l2}.

\begin{lemma}\label{l0}
If $\mathcal{B}$ is a regular base at non-isolated points for a
space $X$, then the family $\mathcal{B}^{m}\subset \mathcal{B}$ is
locally finite at non-isolated points and also covers $X-I$.
\end{lemma}

\begin{lemma}\label{l1}
Let $\mathcal{B}$ be a regular base at non-isolated points for $X$.
If $\mathcal{B}'\subset\mathcal{B}$ is point-finite at non-isolated
points, then
$\mathcal{B}^{\prime\prime}=(\mathcal{B}-\mathcal{B}^{\prime})\cup\mathcal{I}(X)$
is a regular base at non-isolated points for $X$.
\end{lemma}

\begin{lemma}\label{l2}
If $\mathcal{B}$ is a regular base at non-isolated points for $X$,
put$$\mathcal{B}_{1}=\mathcal{B}^{m},\ \
\mathcal{B}_{i}=[(\mathcal{B}-\bigcup_{j=1}^{i-1}\mathcal{B}_{j})\cup\mathcal{I}(X)]^{m},
i=2,3,\cdots .$$ Then
$\mathcal{B}=(\bigcup_{i=1}^{\infty}\mathcal{B}_{i})\cup\mathcal{I}(X)$,
and for each $i\in\mathbb{N}$, $\mathcal{B}_{i}$ is locally finite
at non-isolated points and $\mathcal{B}_{i+1}\cup\mathcal{I}(X)$
refines $\mathcal{B}_{i}\cup\mathcal{I}(X)$.
\end{lemma}

Recall that a topological space $X$ is {\it monotonically
normal} \cite{HLZ} if for each ordered pair $(p, C)$, where $C$ is a
closed set for $X$ and $p\in X-C$, there exists an open subset $H(p,
C)$ satisfying the following conditions:

(i)\ $p\in H(p, C)\subset X-C$;

(ii)\ For every closed subset $D$ for $X$, if $D\subset C$,
then $H(p, C)\subset H(p, D)$;

(iii)\ If $p\not=q\in X$, then $H(p, \{q\})\cap H(q,
\{p\})=\emptyset$.

A $\mbox{T}_{2}$-paracompact space or monotonically normal space is
a collectionwise normal space \cite{HLZ}.

\begin{lemma}\label{l3}
If a space $X$ has a strong development at non-isolated points, then
$X$ is a monotonically normal and paracompact space.
\end{lemma}

\begin{proof}
Let $\{\mathcal{W}_{i}\}_{i\in\mathbb{N}}$ be a strong development
at non-isolated points for $X$, where $\mathcal{W}_{i+1}$ refines
$\mathcal{W}_{i}$ for every $i\in\mathbb{N}$.

(1) Claim.\ Let $A$ be a closed subset for $X$. If $x\in (X-A)\cap (X-I)$,
then there exists an $i\in\mathbb{N}$ such that $\mbox{st}(x,
\mathcal{W}_{i})\cap\mbox{st}(A, \mathcal{W}_{i})=\emptyset$.

In fact, since $X-A$ is an open neighborhood of $x$, there exists a
$j\in\mathbb{N}$ and an open neighborhood $V$ of $x$ such that
$\mbox{st}(V, \mathcal{W}_{j})\subset X-A$. Also, there exists a
$i\geq j$ such that $\mbox{st}(x, \mathcal{W}_{i})\subset V$.
Since $\mbox{st}(A,
\mathcal{W}_{i})\subset X-V$, we have $\mbox{st}(x,
\mathcal{W}_{i})\cap\mbox{st}(A, \mathcal{W}_{i})=\emptyset.$

(2)\ $X$ is a monotonically normal space.

Let $C$ be a closed subset for $X$ and $p\in X-C$. If $p\in I$, then
we let $H(p, C)=\{p\}$; if $p\in X-I$, then there exists a minimum
$n\in\mathbb{N}$ such that $\mbox{st}(p,
\mathcal{W}_{n})\cap\mbox{st}(C, \mathcal{W}_{n})=\emptyset$ by (1),
so we let $H(p, C)=\mbox{st}(p, \mathcal{W}_{n})$. Then $H(p, C)$ is
an open subset for $X$. Clearly this definition of
$H(p, C)$ satisfies the conditions (i) and (ii) in the above
definition of monotonically normal spaces. We next prove that it
also satisfies (iii). In fact, for any distinct points $p, q$ in
$X-I$, fix the $n, m$ for which: $$H(p, \{q\})=\mbox{st}(p, \mathcal{W}_{n}) \mbox{~and~}
H(q, \{p\})=\mbox{st}(q, \mathcal{W}_{m}).$$ Then $$\mbox{st}(p,
\mathcal{W}_{n})\cap\mbox{st}(q, \mathcal{W}_{n})=\emptyset
\mbox{~and~} \mbox{st}(p, \mathcal{W}_{m})\cap\mbox{st}(q,
\mathcal{W}_{m})=\emptyset.$$ By the choice of $n, m$, we have
$n=m$, i.e, $H(p, \{q\})\cap H(q, \{p\})=\emptyset$. Hence it also
satisfies (iii) in the definition of monotonically normal spaces.

(3)\ $X$ is a paracompact space.

Let $\{G_{s}\}_{s\in S}$ be an open cover for $X$ and $S_{0}=\{s\in
S:G_{s}\cap (X-I)\neq\emptyset\}$. Fix a well-order by
``$<$'' on $S_{0}$. For every $i\in\mathbb{N},\ s\in S_{0}$, put

\hspace*{\fill}$F_{s, i}=X-(\mbox{st}(X-G_{s}, \mathcal{W}_{i})
\cup(\bigcup_{s^{\prime}< s}G_{s^{\prime}}))$,\hspace*{\fill}\\
then $F_{s, i}\subset G_{s}$.

(3.1)\ The closed family $\{F_{s, i}\}_{s\in S_{0}, i\in
\mathbb{N}}$ covers $X-I$.

Indeed, for every $x\in X-I$, there exists a minimum $s(x)\in
S_{0}$ such that $x\in G_{s(x)}$. Since
$\{\mathcal{W}_{i}\}_{i\in\mathbb{N}}$ is a strong development at
non-isolated points for $X$, there exists an $i(x)\in\mathbb{N}$
such that $\mbox{st}(x, \mathcal{W}_{i(x)})\subset G_{s(x)}$.  Hence
$x\in F_{s(x), i(x)}$.

(3.2)\ For every $i\in\mathbb{N}$, $\{F_{s, i}\}_{s\in S_{0}}$ is a
discrete and closed family for $X$.

The family $\{F_{s, i}\}_{s\in
S_{0}}$ is disjoint by construction, hence if $x\in I$ then $\{x\}$ is a
neighborhood that intersects $F_{s, i}$ for at most one $s$. If $x\in X\setminus I$ then, using (3.1), $x\in\bigcup_{s\in S_{0}}G_{s}$.
Hence there exists a minimum $s(x)\in S_{0}$ such that $x\in
G_{s(x)}$. Then $G_{s(x)}\cap\mbox{st}(x, \mathcal{W}_{i})$ is an
open neighborhood of $x$. If $s^{\prime}< s(x)$, then $x\in
X-G_{s^{\prime}}$, so we have $$\mbox{st}(x,
\mathcal{W}_{i})\subset\mbox{st}(X-G_{s^{\prime}}, \mathcal{W}_{i})
\mbox{~and~} \mbox{st}(x, \mathcal{W}_{i})\cap F_{s^{\prime},
i}=\emptyset;$$ If $s^{\prime}> s(x)$, then $G_{s(x)}\cap
F_{s^{\prime}, i}=\emptyset$, so there is only one member of
$\{F_{s, i}\}_{s\in S_{0}}$ which meets $G_{s(x)}\cap\mbox{st}(x,
\mathcal{W}_{i})$. Hence $\{F_{s, i}\}_{s\in S_{0}}$ is a discrete
and closed family for $X$.

$X$ is collectionwise
normal since monotonically normal spaces are collectionwise normal \cite{HLZ}.
For every $F_{s, i}$, there exists an open subset $G_{s, i}$ such
that $F_{s, i}\subset G_{s, i}\subset G_{s}$ and $\{G_{s, i}\}_{s\in
S_{0}}$ is a discrete family. Let
$$\mathcal{B}_{i}=\{G_{s, i}\}_{s\in S_{0}} \cup \{\{x\}: x\in
I-\bigcup_{s\in S_{0}}G_{s, i}\}.$$ Then
$\bigcup_{i\in\mathbb{N}}\mathcal{B}_{i}$ is a $\sigma$-locally
finite open cover for $X$ and refines $\{G_{s}\}_{s\in S}$. Since
$X$ is regular, $X$ is paracompact.
\end{proof}

Next we shall prove the main theorems in this section.

\begin{theorem}\label{t0}
A space $X$ has a regular base at non-isolated points if and only if
$X$ has a strong development at non-isolated points.
\end{theorem}

\begin{proof}
Necessity. Since $X$ has a regular base at non-isolated points , $X$
has a regular base at non-isolated points
$\mathcal{B}=(\bigcup_{i\in\mathbb{N}}\mathcal{B}_{i})\cup\mathcal{I}(X)$
satisfying Lemma~\ref{l2}, where $\mathcal{B}_{i}$ is locally finite at
non-isolated points and $\mathcal{B}_{i+1}\cup\mathcal{I}(X)$
refines $\mathcal{B}_{i}\cup\mathcal{I}(X)$ for every
$i\in\mathbb{N}$. Put
$\mathcal{W}_{i}=\mathcal{B}_{i}\cup\mathcal{I}(X)$. We will show that
$\{\mathcal{W}_{i}\}_{i\in\mathbb{N}}$ is a strong development at
non-isolated points for $X$. In fact, for every $x\in X-I$ and each
open neighborhood $U$ of $x$, since $\mathcal{B}$ is regular at
non-isolated points , there exists an open neighborhood $V\subset U$
of $x$ such that the set of all members of $\mathcal{B}$ that meet both $V$
and $X-U$ is finite. We can denote these finite elements by $B_{1},
B_{2}, \cdots ,B_{k}$. Then there exists a $j\in\mathbb{N}$ such
that $\mathcal{B}_{j}\cap\{B_{i}: i\leq k\}=\emptyset$. Hence
$\mbox{st}(V, \mathcal{W}_{j})\subset U$.

Sufficiency. Let $\{\mathcal{W}_{i}\}_{i\in\mathbb{N}}$ be a strong
development at non-isolated points for $X$. By Lemma~\ref{l3},
$X$ is paracompact. For every $i\in\mathbb{N}$, let
$\mathcal{B}_{i}$ be a locally finite open refinement for
$\mathcal{W}_{i}$. Without loss of generality, we may assume
$\mathcal{B}_{i+1}$ refines $\mathcal{B}_{i}$ for every
$i\in\mathbb{N}$. We next prove that
$\mathcal{B}=(\bigcup_{i\in\mathbb{N}}\mathcal{B}_{i})\cup\mathcal{I}(X)$
is a regular base at non-isolated points for $X$. Obviously
$\mathcal{B}$ is a base for $X$. For every $x\in X-I$ and each open
neighborhood $U$ of $x$, there exist an open neighborhood $V$ of $x$
and an $i\in\mathbb{N}$ such that $\mbox{st}(V,
\mathcal{W}_{i})\subset U$. If $j\geq i$, then $$\mbox{st}(V,
\mathcal{B}_{j})\subset\mbox{st}(V,
\mathcal{B}_{i})\subset\mbox{st}(V, \mathcal{W}_{i})\subset U.$$
However, since each $\mathcal{B}_{j}$ is
locally finite, there exists an open neighborhood $W(x)$ of $x$ such
that the set of all members of $\bigcup_{j<i}\mathcal{B}_{j}$ that meet
$W(x)$ is finite. Let$V_{1}=V\cap W(x)$. Then the set of all members
of $\mathcal{B}$ that meet $V_{1}$ and $X-U$ is finite.
\end{proof}

Similar
to definition~\ref{d0}, we say a space $X$ has a {\it
development at non-isolated points} \cite{LL} if there exists a
sequence $\{\mathcal{W}_{i}\}_{i\in\mathbb{N}}$ of open covers for
$X$ such that, for every $x\in X-I$ and each open neighborhood $U$
of $x$, there exist an open neighborhood $V$ of $x$ and an $i\in
\mathbb{N}$ such that $\mbox{st}(V,\mathcal{W}_{i})\subset U$.

\begin{theorem}\label{t1}
A space $X$ has a regular base at non-isolated points if and only if
$X$ is a $T_{2}$-paracompact space with a development at
non-isolated points.
\end{theorem}

\begin{proof}
Necessity. By Lemma~\ref{l3} and Theorem~\ref{t0}, if $X$ has a regular base
at non-isolated points, then $X$ is a $\mbox{T}_{2}$-paracompact
space with a development at non-isolated points.

Sufficiency. Let X be a $\mbox{T}_{2}$-paracompact space with a development $\{\mathcal{W}_{i}\}_{i\in\mathbb{N}}$
at non-isolated points. Since $X$ is a $T_{2}$-paracompact
space, there exists a sequence of open covers
$\{\mathcal{B}_{i}\}_{i\in\mathbb{N}}$ for $X$ such that
$\mathcal{B}_{i+1}$ is a star refinement of
$\mathcal{B}_{i}\wedge\mathcal{W}_{i+1}$ for every $i\in\mathbb{N}$.
We next prove that $\{\mathcal{B}_{i}\}_{i\in\mathbb{N}}$ is a
strong development at non-isolated points for $X$. For every $x\in
X-I$ and every open neighborhood $U$ of $x$, there exists an
$i\in\mathbb{N}$ such that $\mbox{st}(x, \mathcal{W}_{i})\subset U$.
Choose a $V\in\mathcal{B}_{i+1}$ such that $x\in V$. Then
$$\mbox{st}(V,
\mathcal{B}_{i+1})\subset\mbox{st}(x,\mathcal{B}_{i})\subset\mbox{st}(x,
\mathcal{W}_{i})\subset U.$$ By Theorem~\ref{t0}, $X$ has a regular base
at non-isolated points.
\end{proof}

\noindent{\bf Remark} We cannot omit the condition
``$\mbox{T}_{2}$'' in Theorem~\ref{t1}. In fact, let $X$ be the finite
complement topology on $\mathbb{N}$. Then $X$ is a
$\mbox{T}_{1}$-compact and developable space, but it is not a
$\mbox{T}_{2}$-space.

The following corollary is a complement for Lemma~\ref{l2}.

\begin{corollary}
A space $X$ has a regular base at non-isolated points if and only if
$X$ is a regular space with a development at non-isolated points
$\{\mathcal{B}_{i}\cup\mathcal{I}(X)\}_{i\in\mathbb{N}}$, where
$\mathcal{B}_{i}$ is locally finite at non-isolated points for every
$i\in \mathbb{N}$.
\end{corollary}

\begin{proof}
Necessity. It is easy to see by the proof of necessity in Theorems~\ref{t0} and~\ref{t1}.

Sufficiency. Let $X$ be a regular space with a development at
non-isolated points
$\{\mathcal{B}_{i}\cup\mathcal{I}(X)\}_{i\in\mathbb{N}}$, where
$\mathcal{B}_{i}$ is locally finite at non-isolated points for every
$i\in \mathbb{N}$. For each $i\in\mathbb{N}$, let $$U_{i}=\{x\in
X:\mathcal{B}_{i}\mbox{~is locally finite at point~}x\}.$$ Then
$U_{i}$ is an open subset and $\mathcal{B}_{i}$ is locally finite at
each point of $U_{i}$. Since $X-I\subset U_{i}$, $X-U_{i}\subset I$
and $X-U_{i}$ is an open subset for $X$. Hence $U_{i}$ is an open
and closed subset for $X$. Thus $\mathcal{B}_{i}|U_{i}=\{B\cap
U_{i}: B\in\mathcal{B}_{i}\}$ is an open and locally finite family.

By Theorem~\ref{t1}, we only need to prove that $X$ is a paracompact
space. In fact, for every open cover $\mathcal{U}$ of $X$ and each
$i\in\mathbb{N}$, let
$$\mathcal{V}_{i}=\{B\cap U_{i}: B\in\mathcal{B}_{i}\mbox{~and there exists
an~}U\in\mathcal{U}\mbox{~such that~}B\subset U\}$$ and $$V_{i}=\cup\mathcal{V}_{i}.$$ Put
$$\mathcal{V}=(\bigcup_{i\in\mathbb{N}}\mathcal{V}_{i})\cup\{\{x\}:
x\in F\}, \mbox{~where~} F=\bigcap_{i\in\mathbb{N}}(X-V_{i}).$$ Then
$\mathcal{V}$ is a cover for $X$ and $F\subset I$. In fact, if $x\in
X-I$, then there exists an $U\in\mathcal{U}$ such that $x\in U$.
Hence there exists an $n\in\mathbb{N}$ such that $\mbox{st}(x,
\mathcal{B}_{n})\subset U$. Fix a $B\in\mathcal{B}_{n}$ such that
$x\in B$. Then $B\subset U$ and $x\in B\cap
U_{n}\in\mathcal{V}_{n}$. So $x\in V_{n}$. Then $F$ is a closed and
discrete subset for $X$. Hence $\mathcal{V}$ is an open $\sigma$-locally
finite cover and refines $\mathcal{U}$. By the regularity, $X$ is a
paracompact space.
\end{proof}

\begin{example}
There exists a non-regular $\mbox{T}_{2}$-space with a development
at non-isolated points.
\end{example}

Let $\mathbb{Q}$, $\mathbb{P}$ denote the rational numbers and the
irrational numbers, respectively. Let $X=\mathbb{R}$ and endow $X$ with
the following topology \cite{BeL}: every point of $\mathbb{P}$ is an
isolated point; every point $x\in\mathbb{Q}$ has neighborhoods of the
following form:$$B(x, n)=\{x\}\cup\{y\in\mathbb{P}: |y-x|<1/n\},
n\in\mathbb{N}.$$ Then $X$ is a non-regular $\mbox{T}_{2}$-space and
the isolated points set of $X$ is $\mathbb{P}$. We denote
$\mathbb{Q}=\{q_{m}: m\in\mathbb{N}\}$. For any $n,m\in\mathbb{N}$,
let $$\mathcal{B}_{n, m}=\{B(q_{m}, n), \mathbb{R}-\{q_{m}\}\},$$
Then $\mathcal{B}_{n, m}$ is a finite open cover for $X$, and
$\mbox{st}(q_{m}, \mathcal{B}_{n, m}\cup\mathcal{I}(X))=B(q_{m},
n)$. Hence $\{\mathcal{B}_{n,
m}\cup\mathcal{I}(X)\}_{n,m\in\mathbb{N}}$ is a development at
non-isolated points for $X$ and $\mathcal{B}_{n, m}$ is locally
finite for any $n,m\in\mathbb{N}$.

\vskip 0.5cm
\section{Metrization Theorems}
In this section we shall discuss the metrization problems on spaces
with the properties of bases at non-isolated points.

$X$ is called a {\it perfect space} if every open subset of $X$ is
an $F_{\sigma}$-set in $X$.

\begin{theorem}\label{t3}
Let $X$ be a space. Then the following are equivalent:
\begin{enumerate}
\item $X$ is metrizable;

\item $X$ is a perfect space with a regular base at non-isolated
points;

\item $X$ is a perfect space with a strong development at non-isolated
points.
\end{enumerate}
\end{theorem}

\begin{proof}
By Theorems~\ref{t2} and~\ref{t0}, we only need to prove $(3)\Rightarrow(1)$.

Let $X$ be a perfect space with a strong development at the
non-isolated points $\{\mathcal{W}_{i}\}_{i\in\mathbb{N}}$ of $X$.
Then there exists a sequence of open sets
$\{G_{n}\}_{n\in\mathbb{N}}$ such that
$X-I=\bigcap_{n=1}^{\infty}G_{n}$. For every $n\in\mathbb{N}$, let
$\mathcal{U}_{n}=\{G_{n}\}\cup\{\{x\}:x\in I-G_{n}\}$. Then
$\{\mathcal{U}_{n}\}_{n\in\mathbb{N}}$ is a sequence of open covers
for $X$.
Put $\mathcal{V}_{2n-1}=\mathcal{W}_{n}$ and $\mathcal{V}_{2n}=\mathcal{U}_{n}$, for each $n\in \mathbb{N}$.
Then $\{\mathcal{V}_{n}\}_{n\in\mathbb{N}}$ is a strong development for
$X$, and $X$ is metrizable by \cite[Theorem 5.4.2]{ER}.
\end{proof}

\noindent{\bf Remark} By Example~\ref{e0}, we see the condition ``$X$
is perfect'' in (2) and (3) of Theorem~\ref{t3} cannot be omitted, although clearly it can be
replaced with the condition that $I(X)$ is an $F_{\sigma}$-set.

\begin{definition}
Let $\mathcal{B}=\bigcup_{i\in\mathbb{N}}\mathcal{B}_{i}$ be a base
for space $X$. $\mathcal{B}$ is called {\it $\sigma$-locally finite
at non-isolated points}, if for every $i\in\mathbb{N}$,
$\mathcal{B}_{i}$ is locally finite at non-isolated points for
$X$µÄ.
\end{definition}

Similarly, we can  define the notion of spaces with a {\it
$\sigma$-discrete base at non-isolated points}.

\begin{definition}
Let $\mathcal{B}$ be a family of subsets of $X$. For every $x\in X$,
$\mathcal{B}$ is called {\it hereditarily closure-preserving at $x$}
if, for any $H(B)\subset B\in\mathcal{B}$,
$x\in\overline{\cup\{H(B): B\in\mathcal{B}\}}$, then
$x\in\cup\{\overline{H(B)}: B\in\mathcal{B}\}$. $\mathcal{B}$ is
called {\it a hereditarily closure-preserving collection} for $X$
if, for every $x\in X$, $\mathcal{B}$ is hereditarily
closure-preserving at $x$.
\end{definition}

It is easy to verify
that a collection is hereditarily closure preserving if and only if it is hereditarily
closure preserving at non-isolated points.

\begin{lemma}\label{l4}
Let $\mathcal{B}$ be locally finite at non-isolated points for $X$.
Then $\mathcal{B}$ is hereditarily closure-preserving.
\end{lemma}

\begin{proof}
Let $\mathcal{B}=\{B_{\alpha}: \alpha\in\Gamma\}$. For every
$\alpha\in\Gamma, \mbox{choose~}H_{\alpha}\subset B_{\alpha}$. We
can assume $x\in X-I$ and denote
$\mathcal{H}=\{H_{\alpha}\}_{\alpha\in\Gamma}$. If
$x\in\overline{\cup\mathcal{H}}$, then there exists an open
neighborhood $U(x)$ of $x$ such that the set of all members of
$\{H_{\alpha}\}_{\alpha\in\Gamma}$ that meet $U(x)$ is finite because
$\{H_{\alpha}\}_{\alpha\in\Gamma}$ is locally finite at non-isolated
points. we denote these finite elements by $H_{\alpha_{1}},
H_{\alpha_{2}},\cdots, H_{\alpha_{n}}$. Since
$$\overline{\cup\mathcal{H}}=\overline{\cup(\mathcal{H}-\{H_{\alpha_{i}}:
i\leq n\})}\cup\overline{\cup\{H_{\alpha_{i}}: i\leq n\}},
\mbox{~and}$$
$$U(x)\cap (\cup(\mathcal{H}-\{H_{\alpha_{i}}: i\leq
n\}))=\emptyset,$$ we have $x\in\overline{\cup\{H_{\alpha_{i}}:
i\leq n\}}$. Hence $x\in\cup\overline{\mathcal{H}}$.
\end{proof}

\begin{lemma}\cite{BD}\label{l5} A regular space $X$ is metrizable if and only if $X$ has
a $\sigma$-hereditarily closure-preserving base.
\end{lemma}

\begin{lemma} Let $X$ be a regular space. Then the following
conditions are equivalent:
\begin{enumerate}

\item $X$ is metrizable;

\item $X$ has a base which is $\sigma$-discrete  at non-isolated points;

\item $X$  has a base which is $\sigma$-locally finite  at non-isolated
points.
\end{enumerate}
\end{lemma}

\begin{proof}
It is easy to see by Theorem~\ref{t2}, Lemmas~\ref{l4} and~\ref{l5}
\end{proof}

Let $X$ be a topological space and $\tau(X)$ its topology. $g:\mathbb{N}\times X\rightarrow
\tau (X)$ is called a $g$-function if, for any $x\in X$ and $n\in
\mathbb{N}$, $x\in g(n, x)$. A space $X$ is called a {\it $\beta$-space} \cite{Ho72} if there
exists a $g$-function such that, for every $x\in X$ and sequence
$\{x_{n}\}$ in $X$, if $x\in g(n, x_{n})$ for each $n\in\mathbb{N}$,
then $\{x_{n}\}$ has a cluster point in $X$.
Obviously every developable space is a $\beta$-space.

\begin{theorem}\label{t5}
A space $X$ is metrizable if and only if $X$ is a $\beta$-space with
a regular base at non-isolated points.
\end{theorem}

\begin{proof}
We only need to prove the sufficiency. Let $X$ be a $\beta$-space
with a regular base at non-isolated points. By Theorem~\ref{t3}, it
suffices to prove that $I(X)$ is an $F_\sigma$-set. Suppose $g$ is a
$g$-function satisfying the above definition of $\beta$-spaces.
Since $X$ has a regular base at non-isolated points, $X$ has a
regular base at non-isolated points
$\mathcal{B}=(\bigcup_{n\in\mathbb{N}}\mathcal{B}_{n})\cup\mathcal{I}(X)$
satisfying Lemma~\ref{l2}, where $\mathcal{B}_{n}$ is locally finite at
non-isolated points and $\mathcal{B}_{n+1}\cup\mathcal{I}(X)$
refines $\mathcal{B}_{n}\cup\mathcal{I}(X)$ for each
$n\in\mathbb{N}$. For each $n\in\mathbb{N}$ and $x\in X-I$, put
$$b(n, x)=\cap\{B\in\mathcal{B}_{n}: x\in B\}.$$
Then $\{b(n, x)\}_{n\in\mathbb{N}}$ is a local base for $x\in X-I$.
For each $n\in\mathbb{N}$, put $$h(n, x)=(\cap\{g(i, x): i\leq
n\})\cap b(n, x), ~x\in X-I;$$
$$H_{n}=\cup\{h(n, x): x\in X-I\}.$$
Then $X-I\subset H_{n}$ and $H_{n}$ is an open subset for $X$. We
next prove $X-I=\bigcap_{n\in\mathbb{N}}H_{n}$. Let
$x\in\bigcap_{n\in\mathbb{N}}H_{n}$. Then there exists some point
$x_{n}\in X-I$ such that $x\in h(n, x_{n})$ for each
$n\in\mathbb{N}$. Since $X$ is a $\beta$-space and $x\in g(n,
x_{n})$, $\{x_{n}\}$ has a cluster point in $X$. Let $y$ be a
cluster point of $\{x_{n}\}$. Then $y\in X-I$ and $b(n, y)$ is an
open neighborhood of $y$. Without loss of generality, we can assume
$x_{n_{i}}\in b(i, y)$ for each $i\in\mathbb{N}$. We will show that $b(i,
x_{n_{i}})\subset b(i, y)$. If not, choose a point $z\in b(i,
x_{n_{i}})-b(i, y)$, then there exists a $B\in\mathcal{B}_{i}$ such
that $y\in B$ and $z\notin B$. Since $x_{n_{i}}\in b(i, y)\subset
B$, $z\in b(i, x_{n_{i}})\subset B$, a contradiction. Hence
$$x\in\bigcap_{i\in\mathbb{N}}h(n_{i},
x_{n_{i}})\subset\bigcap_{i\in\mathbb{N}}h(i, x_{n_{i}})
\subset\bigcap_{i\in\mathbb{N}}b(i, y)=\{y\},$$ i.e, $x=y\in X-I.$
Thus $X-I=\bigcap_{n\in\mathbb{N}}H_{n}$, and $I$ is an
$F_{\sigma}$-set for $X$. By Theorem~\ref{t3}, $X$ is metrizable.
\end{proof}

\noindent{\bf Remark} The Stone-\v{C}ech compactification
$\beta\mathbb{N}$ of $\mathbb{N}$ is a $\beta$-space, but it is not
a perfect space \cite[Corollary 3.6.15]{ER}; Sorgenfrey line is a
perfect space, but it is not a $\beta$-space \cite[Example
4.4]{Ho72}. Hence, Theorem~\ref{t3} and Theorem~\ref{t5} are independent each
other.

\vskip 0.5cm
\section{Relations with Generalized Metrizable Spaces}
\begin{definition}\cite{WF}
Let $X$ be a topological space and let $A$ be a subset of $X$. The {\it discretization} of $X$ by $A$
is the space whose topology is generated by the base
$\{U: U \mbox{~is an open subset of~} X\}\cup\{\{x\}:x\in
A\}$.
It is denoted by
$X_A$ in \cite[Example 5.1.22]{ER}. We say that a space $Y$ is a {\it discretization} of $X$ if $Y=X_A $ for some $A\subset X$.
\end{definition}

\begin{theorem}\label{t6}
Let $X$ be a metric space. If $A\subset X$ and $X_{A}$ is the discretization
of $X$ by $A$, then $X_{A}$ has a regular base at
non-isolated points.
\end{theorem}

\begin{proof}
Since $X$ is a metric space, $X$ has a regular base
$\mathcal{B}_{1}$. Let $\mathcal{B}=\mathcal{B}_{1}\cup\{\{x\}:x\in
A\}$. Obviously, $\mathcal{B}$ is a regular base at non-isolated
points for $X_{A}$.
\end{proof}

\noindent{\bf Remark} If a space $X$ with a regular base at
non-isolated points, then is it a discretizable space of a metric
space? The answer is negative, see Example~\ref{e1}. Recall that $X$ is said to have
a $G_{\delta}$-{\it diagonal} if there exists a sequence
$\{\mathcal{U}_{n}\}_{n\in\mathbb{N}}$ of open covers such that
$\{x\}=\bigcap_{n\in\mathbb{N}}\mbox{st}(x, \mathcal{U}_{n})$ for
every $x\in X$.

\begin{example}\label{e1}
There exists a space $Y$ having a regular base at non-isolated
points. However, $Y$ is not a discretization of a metric space.
\end{example}

Let $X$ be the Michael line in Example~\ref{e0}, and denote it by
$X_{B}$. Let $X^{\ast}$ be a copy of $X_{B}$ and
$f:X_{B}\rightarrow X^{\ast}$ a homeomorphic map. Put
$Z=X_{B}\bigoplus X^{\ast}$ and let $g:Z\rightarrow Y$ be a quotient
map by identifying $\{x, f(x)\}$ to a point for each $x\in
X_B\setminus B$ in $Z$. Then $Y$ is a quotient space.

By \cite{PO}, it is easy to see $Y$ has no $G_{\delta}$-diagonal.
Since the discretization of a metric space has a
$G_{\delta}$-diagonal, $Y$ is not a discretization of a metric space. We next prove that $Y$ has a regular base at non-isolated
points.

Put $\mathcal{I}=\{\{x\}: x\in B\}$ and let $\mathcal{B}$ be a
regular base of $\mathbb I$ with the Euclidean topology. Then
$\mathcal{B}\cup\mathcal{I}$ is a regular base at non-isolated
points for $X_{B}$. Hence $f(\mathcal{B})\cup f(\mathcal{I})$ is a
regular base at non-isolated points for $X^{\ast}$. Then
$\mathcal{G}=\{g(B\cup f(B)): B\in \mathcal{B}\}\cup\mathcal{I}\cup
f(\mathcal{I})$ is a
 regular base at non-isolated
points for $Y$.

Indeed, it is easy to see that $\mathcal{G}$ is a base for $Y$. For
every $y\in Y-I(Y)$ and each open neighborhood $U$ of $y$ in $Y$,
there exists a point $x\in X_{B}$ such that $g(x)=y$. Then
$g(f(x))=y$, and $x, f(x)\in g^{-1}(U)$. Since
$$\mathcal{B}_{0}=\mathcal{B}\cup f(\mathcal{B})\cup \mathcal{I}\cup
f(\mathcal{I})$$ is a regular base at non-isolated points for $Z$,
there exist open neighborhoods $V_{x}, V_{f(x)}\subset g^{-1}(U)$ of
$x, f(x)$ in $Z$ respectively such that the set of all members of
$\mathcal{B}_{0}$ that meet $V_{x}$ and $Z-g^{-1}(U)$ is finite, and the
set of all members of $\mathcal{B}_{0}$ that meet $V_{f(x)}$ and
$Z-g^{-1}(U)$ is also finite. Since $f$ is a homeomorphic map, there
exists a $B\in \mathcal{B}$ such that $x\in B\subset V_{x}$ and
$f(x)\in f(B)\subset V_{f(x)}$. Then $g(x)=y\in g(B\cup f(B))\subset
U$. Since the set of all members of $\mathcal{B}_{0}$ that meet $B\cup
f(B)$ and $Z-g^{-1}(U)$ is finite. If $V\in\mathcal{B}_{0}$, then
$g^{-1}(g(V))=V$, hence the set of all members of $\mathcal{G}$ that
meet $g(B\cup f(B))$ and $Y-U$ is finite. Thus $Y$ has a regular
base at non-isolated points.

\begin{definition}\cite{WF}
An {\it ortho-base} $\mathcal{B}$ for $X$ is a base of $X$ such that
either $\cap \mathcal{A}$ is open in $X$ or $\cap
\mathcal{A}=\{x\}\notin\mathcal I(X)$ and $\mathcal{A}$ is a
neighborhood base at $x$ in $X$ for each $\mathcal{A}\subset
\mathcal{B}$. A space $X$ is a {\it proto-metrizable space} if it is
a paracompact space with an ortho-base.
\end{definition}

Recall that a space $X$ is called a {\it $\gamma$-space} if there exists a $g$-function $g(n, x)$ for
$X$ satisfying for each $x\in X$ and sequences $\{x_{n}\},
\{y_{n}\}$ if $x_{n}\in g(n, y_{n})$ and $y_{n}\in g(n, x)$ for each
$n\in\mathbb N$, then $x_{n}\to x$.

\begin{theorem}\label{t7}
If a space $X$ has a regular base at non-isolated points, then $X$ is:
\begin{enumerate}
\item a proto-metrizable space, and

\item a $\gamma$-space.
\end{enumerate}
\end{theorem}

\begin{proof}
(1) By Lemma~\ref{l3} and Theorem~\ref{t0}, $X$ is a
paracompact space. Also, $X$ has an ortho-base by \cite[Theorem
3.4]{LL}. Hence $X$ is a proto-metrizable space.

(2) To prove part (2), for each $n\in\mathbb{N}\mbox{~and~} x\in X$ define a function
$g:\mathbb{N}\times X\to\tau(X)$ as follows: if $x\in I$, then $g(n,
x)=\{x\}$; if $x\in X-I$, then $g(n, x)=b(n, x)$, where $b(n, x)$ is
the same as in the proof in Theorem~\ref{t5}. Then $\{g(n,
x)\}_{n\in\mathbb{N}}$ is a decreasing and open neighborhood base of
$x$, and if $y\in g(n, x)$, then $g(n, y)\subset g(n, x)$. For each
$x\in X$ and sequences $\{x_{n}\}, \{y_{n}\}$, if $x_{n}\in g(n,
y_{n})$ and $y_{n}\in g(n, x)$ for each $n\in \mathbb{N}$, then
$x_{n}\in g(n, y_{n})\subset g(n, x)$, thus $x_{n}\to x$. Hence $X$
is a $\gamma$-space.
\end{proof}

\begin{example}
There exists a proto-metrizable space which has no regular base at
non-isolated points.
\end{example}

The proto-metrizable but non-$\gamma$-space described in Section 3 in \cite{GZ} works.

\noindent{\bf Remark} From the discussion above, it can be seen that
spaces with a regular base at non-isolated points are strictly
between the discretizations of metrizable spaces and
proto-metrizable spaces.

\begin{corollary}\label{c1}
Let $X$ have a $G_{\delta}$-diagonal. Then the following conditions
are equivalent:
\begin{enumerate}
\item $X$ is a discretizations of a metrizable space;

\item $X$ has a regular base at non-isolated points;

\item $X$ is a proto-metrizable space.
\end{enumerate}
\end{corollary}

\begin{proof}
By Theorems~\ref{t6} and~\ref{t7}, we have $(1)\Rightarrow(2)\Rightarrow(3)$.
By \cite[Theorem 3.1]{GZ}, it can be obtained
$(3)\Rightarrow(1)$.
\end{proof}

The condition ``$G_{\delta}$-diagonal'' cannot be omitted in Corollary~\ref{c1} by Example~\ref{e1}.

\begin{question}\label{q1}
Under what conditions  a proto-metrizable space has a regular base
at non-isolated points?
\end{question}

\noindent{\bf Remark} Since a proto-metrizable space is a
paracompact space, Theorem~\ref{t1} is an answer for Question~\ref{q1}.
However, we expect a simpler answer.

\begin{definition}
Let $\mathcal{B}$ be a base of a space $X$. $\mathcal{B}$ is {\it
point-regular} \cite{Al}~({\it point-regular at non-isolated points}
\cite{LL}, resp.) for $X$, if for each (non-isolated, resp.) point
$x\in X$ and $x\in U$ with $U$ open in $X$, $\{B\in\mathcal{B}: x\in
B\not\subset U\}$ is finite.
\end{definition}

Obviously, every regular base at non-isolated points is a
point-regular base at non-isolated points. In \cite{LL}, it is
proved that a space $X$ has a point-regular base at non-isolated
points if and only if $X$ is an open, boundary-compact image of a
metric space. On the other hand, a space $X$ is an open,
boundary-compact, $s$-image of a metric space if and only if $X$ has
a point-countable base which is point-regular at non-isolated
points. The following question is posed in \cite[Question 5.1]{LL}:

\begin{question}\label{q2}\cite[Question 5.1]{LL}
Let a space $X$ have a point-countable base. If $X$ has a
point-regular base at nos-isolated points, is $X$ an open,
boundary-compact, $s$-image of a metric space?
\end{question}

Next, we give an affirmative
answer for Question~\ref{q2}.

A space $X$ is called {\it metalindel\"{o}f}\,\, if every open cover
of $X$ has a point-countable open refinement.

\begin{theorem}\label{t8}
The following are equivalent for a space $X$:
\begin{enumerate}
\item $X$ has a point-countable base, and has a point-regular base at
non-isolated points;

\item $X$ has a point-countable base which is point-regular at
non-isolated points;

\item $X$ is an open boundary-compact, $s$-image of a metric
space;

\item $X$ is an open $s$-image of a metric space, and is an open
boundary-compact image of a metric space;

\item $X$ is a metalindel\"{o}f space with a point-regular base at
non-isolated points.
\end{enumerate}
\end{theorem}

\begin{proof}
It is proved in \cite{LL} that if $\mathcal{P}$ is a point-regular
base at non-isolated points for a space $X$, then we can assume that
$\mathcal{P}=\bigcup_{n\in\mathbb{N}}\mathcal{P}_n$ satisfies the
following conditions:

(a)\ $\mathcal{P}_n$ is an open cover and is point-finite at
non-isolated points;

(b)\ $\{\mathcal{P}_n\}$ is a development at non-isolated points for
$X$.

\bigskip

$(1)\Rightarrow(2)$. Suppose that $X$ has a point-countable base
$\mathcal{B}$, and suppose that $X$ has a point-regular base at non-isolated points
$\mathcal{P}$. We can assume that
$\mathcal{P}=\bigcup_{n\in\mathbb{N}}\mathcal{P}_n$ satisfies the
conditions (a) and (b). For each $n\in\mathbb{N}$, put

\hspace{2cm}$\mathcal{B}'=\{B\in\mathcal{B}: B\not\subset I(X)\}$;

\hspace{2cm}$\mathcal{V}_n(B)=\{P\in\mathcal{P}_n: B\subset
P\}$,\,\, $\forall\, B\in\mathcal{B}'$;

\hspace{2cm}$\hat{P}=\cup\{B\in\mathcal{B}':
P\in\mathcal{V}_n(B)\}$,\,\, $\forall\, P\in\mathcal{P}_n$;

\hspace{2cm}$\hat{\mathcal{P}}_n=\{\hat{P}: P\in\mathcal{P}_n\}$.

\noindent Then $\hat{\mathcal{P}}_n$ is point-countable. In fact, if
$x\in\hat{P}\in\hat{\mathcal{P}}_n$, then there is $B'\in\mathcal{B}'$
such that $x\in B'$ and $P\in\mathcal{V}_n(B')$. Since
$\{B\in\mathcal{B}': x\in B\}$ is countable, and each
$\mathcal{V}_n(B)$ is finite for each $B\in\mathcal{B}'$ by the
condition (a), it follows that $\{P\in\mathcal{V}_n(B): x\in
B\in\mathcal{B}'\}$ is countable.

Put
$$\hat{\mathcal{P}}=(\bigcup_{n\in\mathbb{N}}\hat{\mathcal{P}}_n)\cup\mathcal{I}(X).$$
Then $\hat{\mathcal{P}}$ is point-countable. If $x\in U-I$ with $U$
open in $X$, then there is $m\in\mathbb{N}$ such that $x\in\mbox{st}(x,
\mathcal{P}_m)\subset U$ by the condition (b). Take
$P\in\mathcal{P}_m$ with $x\in P$, then there is $B\in\mathcal{B}'$
such that $x\in B\subset P$, thus $P\in\mathcal{V}_m(B)$, and $x\in
B\subset\hat{P}\subset P\subset U$. So $\hat{\mathcal{P}}$ is a base
for $X$. Finally, it is easy to see that $\hat{\mathcal{P}}$ is
point-regular at non-isolated points by $\hat{P}\subset P$ for each
$P\in\mathcal{P}$.

$(2)\Rightarrow(3)$ by \cite[Corollary, 3.2]{LL}.
$(3)\Rightarrow(4)$ is obvious. And $(4)\Rightarrow(5)$ by
\cite[Theorem, 3.1]{LL}.

$(5)\Rightarrow(1)$.
Let $X$ be a metalindel\"{o}f space with a
point-regular base at non-isolated points. As in the proof of $(1)\Rightarrow (2)$, there is a sequence $\{\mathcal{P}_{n}\}$ of open covers of $X$ such that
$\{\mathcal{P}_n\}$ is a development at non-isolated points for $X$.
For each $n\in\mathbb{N}$, let $\mathcal{B}_n$ be a point-countable
open refinement of $\mathcal{P}_n$. And put
$$\mathcal{B}=(\bigcup_{n\in\mathbb{N}}\mathcal{B}_n)\cup\mathcal{I}(X).$$
Then $\mathcal{B}$ is a point-countable base for $X$. In fact, if a
non-isolated point $x\in U$ with $U$ open in $X$, then there is
$n\in\mathbb{N}$ such that $\mbox{st}(x, \mathcal{P}_n)\subset U$.
Take $B\in\mathcal{B}_n$ with $x\in B$, then $x\in
B\subset\mbox{st}(x, \mathcal{B}_n)\subset\mbox{st}(x,
\mathcal{P}_n)\subset U$.
\end{proof}

By Theorem~\ref{t8}, the following is obtained.

\begin{corollary}
Every space with a regular base at non-isolated points has a
point-countable base.
\end{corollary}

\vskip0.9cm

\end{document}